\documentclass[10pt]{article}
\usepackage{amsmath,amssymb}
\usepackage{graphicx}

\newtheorem{thm}{Theorem}
\newtheorem{co}[thm]{Corollary}
\newtheorem{lem}[thm]{Lemma}
\newtheorem{pr}[thm]{Proposition}
\newtheorem{definition}[thm]{Definition}
\newenvironment{de}[1]{\begin{definition}[#1]}{\end{definition}}

\newcommand{\dd}{\mathrm{d}}
\newcommand{\D}{\mathbb{D}}
\newcommand{\N}{\mathbb{N}}
\newcommand{\R}{\mathbb{R}}
\newcommand{\Z}{\mathbb{Z}}
\newcommand{\PP}{\mathbb{P}}
\newcommand{\vol}{\mathrm{Vol}\,}

\newcommand{\HNF}{\mathrm{HNF}\,}
\newcommand{\diag}{\mathrm{diag}\,}

\newcommand{\Section}[1]{\section{#1}}

\newcommand{\openbox}{\leavevmode
  \hbox to.77778em{%
    \hfil\vrule
  \vbox to.675em{\hrule width.6em\vfil\hrule}%
  \vrule\hfil}}
\newcommand{\proofname}{Proof}

\newenvironment{proof}[1][\proofname]{\par\normalfont
  \trivlist\item[\hskip\labelsep\itshape #1:]\ignorespaces
  }{\hspace*{1cm}\hspace*{\fill}\openbox \medskip\endtrivlist}

\title{Natural Density Distribution of Hermite Normal Forms of Integer Matrices
  \footnote{To appear in Journal of Number Theory}
}%
\date{February 23, 2011}
\author{G\'erard Maze\\
{\small {\em e-mail:\/} gerard.maze@math.uzh.ch \vspace{-1mm} }\\
{\small Mathematics Institute\vspace{-1mm}}\\
{\small University of Z\"urich\vspace{-1mm}}\\
{\small Winterthurerstr 190, CH-8057 Z\"urich, Switzerland }
\vspace{3mm} }

\begin{document}\maketitle
\thispagestyle{empty}
\begin{abstract}
  The Hermite Normal Form (HNF) is a canonical representation of
  matrices over any principal ideal domain. Over the integers, the
  distribution of the HNFs of randomly looking matrices is far from
  uniform. The aim of this article is to present an explicit
  computation of this distribution together with some
  applications. More precisely, for integer matrices whose entries are
  upper bounded in absolute value by a large bound, we compute the
  asymptotic number of such matrices whose HNF has a prescribed
  diagonal structure. We apply these results to the analysis of some
  procedures and algorithms whose dynamics depend on the HNF of
  randomly looking integer matrices.
\end{abstract}

\vspace{3mm}
\vspace{3mm}

\Section{Introduction}

Given a principal ideal domain $R$, the notion of Hermite Normal Form
(HNF) of a $n \times m$ matrix with entries in $R$ is well
defined. When $R=\Z$, which will be the case in this article, a matrix
in HNF can be defined as follows, see e.g. \cite{cohen, macduffee}:

\begin{definition}[Hermite Normal Form (HNF)]
  A $n \times m$ matrix $H$ with integer entries is in Hermite normal
  form if $H$ is upper triangular with the following properties:
  \begin{enumerate}
  \item The first $r$ rows of $H$ are the non-zero rows of $H$,
  \item for each row $i$, if $h_{ij_i}$ is its first non-zero entry,
    then  $h_{ij_i}>0$ and $j_1<j_2<...<j_r$,
  \item for each $1 \leq k < i \leq r$, the entries $h_{kj_i}$ of the
    $j_i^{th}$ column of $H$ satisfy $0 \leq h_{kj_i} < h_{ij_i}$.
  \end{enumerate}
  The positive integers $h_{ij_i}$ are called the pivot of the matrix
  in HNF.
\end{definition}
A row matrix is,  up to a sign, already in HNF.  We will assume in the
sequel that $n$ is at least  2.  The main result about HNF, discovered
by Charles Hermite,  is that for all $n \times  m$ integer matrix $A$,
there exists  a (possibly non-unique)  unimodular $n \times  n$ matrix
$U$ (i.e., $U\in  \mbox{GL}_n(\Z)$) and a unique $n  \times m$ integer
matrix $H$ in HNF such  that $A=UH$.  The left equivalence between $A$
and $H$  means that there is  a sequence of  elementary row operations
that will produce $H$ when applied to $A$. Note that the definition of
the HNF can  slightly change in the literature  (e.g. lower triangular
vs. upper triangular, column operations vs. row operations). Since the
matrix  $H$  is  uniquely  defined,  we can  write  without  ambiguity
$H=\HNF(A)$.  Typically, the  shape of  a matrix  in HNF  will  be the
following:

\begin{center}
\begin{tabular}{ccccccccc}
* & * & * & * & * & * & * & * & *\\
0 & * & * & * & * & * & * & * & *\\
0 & 0 & 0 & * & * & * & * & * & *\\
0 & 0 & 0 & 0 & * & * & * & * & *\\
0 & 0 & 0 & 0 & 0 & 0 & 0 & * & *\\
0 & 0 & 0 & 0 & 0 & 0 & 0 & 0 & *\\
0 & 0 & 0 & 0 & 0 & 0 & 0 & 0 & 0\\
0 & 0 & 0 & 0 & 0 & 0 & 0 & 0 & 0\\
\end{tabular}
\end{center}
\vspace{3mm}

The above example is given with $(n,m,r)=(8,9,6)$ and the sequence
$j_i$ of column positions of the pivots is $1,2,4,5,8,9$. As a matter
of fact, only a very small proportion of HNFs of integer matrices has
this type of shape. Anyone who had to compute the HNF of arbitrary
integer matrices more than once was forced to observe that they do not
appear ``randomly'', that is, the elements $h_{ij_i}$ do not seem to
follow an equiprobable law of distribution. For instance, the case
$j_i=i$ and $r=\min(n,m)$ appears predominantly, and a strongly
recurring structure is that all the pivots $h_{ij_i}$ with $i<r$ are
small and increasing with $i$ (typically less than 10, even for
matrices with very large entries) and the last pivot $h_{ij_r}$ is
large (of the order of $\det A$ when $n=m$). This particular point is
intrinsically interesting, but was also used in several occasions (see
e.g. \cite{buchmann, miccancio, pernet}) in order to heuristically
understand or analyze the behavior of an algorithm.

Our focus in this paper is set on the ``probability'' that the HNF of
a random $n \times m$ integer matrix has a given diagonal. We aim to
obtain an explanation of the strong biases mentioned above. For
instance, Proposition \ref{pr1} below gives the frequency of
appearance of a given non-zero diagonal in the HNF of a randomly
looking matrix. Proposition \ref{pr1} also shows that the density of
HNFs with the above shape is in fact $0$. Corollary \ref{co} shows
that given strictly positive integers $d_1,d_2,...,d_{n-1}$, $n\geq2$,
the ``probability'' that a $n \times n$ integer matrix $A$ has a HNF
of the form
\[
\HNF(A)=
\left[
\begin{array}{ccccc}
d_1 & * & * & * & *\\
0 & d_2 & * & * & *\\
0 & 0 & \ddots & * & *\\
0 & 0 & 0 & d_{n-1} & *\\
0 & 0 & 0 & 0 & d\\
\end{array}
\right]
\]
where $d= \frac{\det(A)}{\prod_{i=1..n-1} d_i}$ is given by
\[
\left( \zeta(n) \cdot \zeta(n-1) ...  \cdot \zeta(2)
\cdot d_1^n \cdot d_2^{n-1}  \cdot \ldots  \cdot d_{n-1}^{2}  \right)^{-1},
\]
where $\zeta$ is the usual zeta function. Of course, the notion of
``probability'' and ``density'' used here have to be made precise. The
appropriate concept is the notion of natural density, see
e.g.~\cite{tenenbaum}. Different definitions of densities appear
naturally in analytic number theory with the study of prime numbers
and expected values of arithmetic functions, see e.g. \cite{tenenbaum}
and \cite{hardy} for several examples. As for the {\it natural
  density}, the explicit multidimensional aspects of the question
appear in e.g. \cite{hetzel,martin,maze} and more implicitly in
e.g. \cite{buchmann, cesaro2, cesaro3,hafner,lehmer}. On the more
specialized study of density of canonical form of matrices, let us
mention the work of Evans \cite{evans} where the density of Smith
normal form over the ring of integers of a local field is studied. The
subject treated in the present article does not seem to have been the
object of a publication in the past.

The article is structured as follows. We address the question of a
suitable definition of natural density in $\Z^k$ in Section
\ref{natural_density} below. In Section \ref{unimodular} we present
some results linking unimodular matrices and natural density of
vectors. The main results of the article are stated and proved in
Section \ref{main} and in Section \ref{application} we present some
applications.

We will used the following notations. The set of primes in $\N^*=\N
\setminus \{0\}$ is $\PP$, Landau's notations $f(x)=o(x)$ and
$g(x)=O(x)$ mean that $\lim_{x \rightarrow \infty} f(x)/x=0$ and
$\limsup_{x \rightarrow \infty} |g(x)/x| < \infty$. The Riemann zeta
function $\zeta$ is $\zeta(s) = \sum_{n\geq 1} n^{-s}=\prod_{p\in\PP}
(1-p^{-s})^{-1}$. The cardinality of a set $S$ is $|S|$. We will also
use the expression ``randomly looking (integer) vector'' in an
informal way, meaning that the entries of the vector have been chosen
uniformly at random in a large interval $\lbrack -B,B\lbrack$. The
symbol $*$ represents an integer whose value is unimportant depending
on the context.

\section{Natural density in $\Z^k$}\label{natural_density}

In order to make the intuitive notion of probability in $\Z^k$ precise
we first remark that the uniform distribution over $\Z^k$ or over
$\N^k$, even when $k=1$, has little meaning. For this reason
researchers often use the concept of {\em natural density} when
stating probability results in $\N$.  In the following we briefly
explain this concept. Let $S\subset \N$ be a set. Define the upper
(respectively lower) natural density as
\[
\overline{\D}(S)=\limsup_{B \rightarrow \infty} \frac{|S \cap
\left[0,B\right[|}{B} \; \; , \;\;
\underline{\D}(S)=\liminf_{B \rightarrow \infty} \frac{|S \cap
\left[0,B\right[|}{B}.
\]
When both limits are equal one defines the natural density of the
set $S$ as
\[
  \D(S):=\overline{\D}(S)=\underline{\D}(S).
\]
The notion of natural density allows to tackle questions related to
the frequency of realization of events concerning randomly looking
integers, i.e., for uniformly chosen integers in $\lbrack 0,B
\lbrack$, when $B$ goes to infinity. A famous example in $\N$ is that
the natural density of square free integers is $6/\pi^2 =
\zeta(2)^{-1}$, see e.g.  \cite{hardy}. The extension of the above
definition in higher dimension is sometimes implicit in the
literature. For example the natural density of $n$ coprime integers,
equal to $\zeta(n)^{-1}$, has been studied by several authors,
starting with Ces\`aro in 1884 \cite{cesaro3} (1881 for the case $n=2$
\cite{cesaro1,cesaro2}), Lehmer in 1900 \cite{lehmer} and Nymann
\cite{nymann}. For the historical fatherhood of the result see
\cite{maze}. This natural density means that there are
$\zeta(n)^{-1}\cdot B^n + o(B^n)$ $n$-vector in $\lbrack 0,B\lbrack^n$
whose entries are coprime. An explicit definition of a higher
dimensional notion of natural density has been developed in
e.g. \cite{hetzel,martin,maze}. In these articles, the notion of
natural density of a set $S$ in $\Z^k$ is defined as a ``centered
symmetric cube'' version of the unidimensional definition, i.e., as
the limit, when it exists, $D(S) =\lim_{B \rightarrow \infty} \frac{ |
  S \cap \lbrack-B,B\lbrack^k | }{ (2B)^k }$. We will however need a
stronger definition. In order to see why, let us consider a set $S$ in
$\N$ with density $\delta$. Since $| S \cap \lbrack 0,l\lbrack| =
\delta \cdot l + o(l)$, any interval $\lbrack l,l+B \lbrack$ with
$l=o(B)$ contains $\delta \cdot B + o(B)$ elements of $S$. Being able
to estimate the local density in non centered cubes does not seem to
be always possible in dimension $k>1$ with the above definition of
density. In order to achieve this, we require in the definition that
the cubes can lie anywhere in $\Z^k$. In the sequel, we call a cube
any set of the form $\prod_{i=1}^k \lbrack z_i-B,z_i+ B\lbrack^k$ for
some $z \in \Z^ k$ and $B>0$.

\begin{de}{Natural density in $\Z^k$}\label{d1}
Let $S$ be a subset of $\Z^k$. If for all $z \in \Z^k$, the following
limit exists
\[
\D(S) = \lim_{B \rightarrow \infty} \frac{|S \cap
\prod_{i=1}^k \left[z_i-B,z_i+B\right[|}{(2B)^k}
\]
and is independent of $z$, then it is called the natural density of
$S$.
\end{de}

Let us notice that it would have been even possible to extend the
definition of natural density from $\Z$ to $\Z^k$ by using
$k$-rectangles instead of cubes (i.e. different $B_i$ for each
dimension). However both definition are equivalent since rectangles
can be decomposed into smaller cubes. We will not use this property in
the sequel. Another direction of generalization is the spherical
model. This setting considers centered $n$-balls instead of
$n$-cubes. Due to the symmetry of the balls around the origin, it is a
natural choice in the study of different asymptotic results concerning
lattices, integer matrices, and varieties in general see
e.g. \cite{nguyen, sarnak}. This model suffers however from the same
problem as noted before and from the fact that the entries of the
different objects of study are not independent anymore, i.e., the
``random looking aspect'' is somehow lost.

In order to prove our main results, we will need the existence and the
value of the natural density of tuples of integers whose greatest
common divisor is a given positive integer $d$. This is treated in
Lemma \ref{l1} below. As mentioned in the introduction, in the weaker
form of density definition given above, this problem has been studied
by several authors, see e.g. \cite{cesaro3,lehmer}.

\begin{lem}\label{l1}
When $k\geq2$, the set $\{(x_1,\ldots,x_k) \in \Z^k \, : \,
\gcd(x_i)=d\}$ has a density equal to $(\zeta(k)\cdot d^k)^{-1}$.
\end{lem}

\begin{proof}
Let $x \in \Z^k$ and $S=\{(x_1,\ldots,x_k) \in \Z^k \, : \,
\gcd(x_i)=d\}$. Then $x \in S$ if and only if $x_i/d \in \Z$ and
$\gcd(x_i/d) = 1$. Let $z'_i = z_i/d$, $B'=B/d$ and
$S'=\{(x_1,\ldots,x_k) \in \Z^k \, : \, \gcd(x_i)=1\}$. The first
equality of the following equations is straightforward.
\begin{equation}\label{e1}
\left| S \cap
\prod_{i=1}^k \left[z_i-B,z_i+B\right[ \right| = 
\left| S' \cap \prod_{i=1}^k \left[z'_i-B',z'_i+B'\right[ \right| =
    \left| S' \cap \prod_{i=1}^k \left[-B',B'\right[ \right| + o(B^k).
\end{equation}
In order to prove that the second equality of Eq. (\ref{e1}) is
valid, consider an element $(x_1,\ldots,x_k)$ in the set of the left
hand side of the equality. Let us fix all the components but the
$i^{th}$ one, and consider $t=\gcd_{j\neq i}(x_j)$. The integers
$x_1,\ldots,x_k$ are coprime if and only if $x_i$ has no common factor
with $t$. So if $P$ is the set of prime divisors of $t$, both the
interval $\lbrack z'_i-B',z'_i+B'\lbrack$ and $\lbrack -B',B'\lbrack$
contains $ 2B' \prod_{p\in P}\left( 1-\frac{1}{p}\right) +o(2B')$
integer coprime to the fixed $x_j$. This shows that the error
resulting in setting $z_i=0$ in the mid term of Eq. (\ref{e1})
can be adjusted by $o(B')=o(B)$. Taking into account the effect of all
dimensions together leads to the correction term $o(B^k)$. Now, as
mentioned before, see e.g.  \cite{lehmer,maze}, we have
\[
\left| S' \cap \prod_{i=1}^k \lbrack-B',B' \lbrack \right| = \zeta(k)
^{-1} (2B')^k + o((B')^k)
\]
which leads to
\[
\D(S) = \lim_{B \rightarrow \infty} \frac{|S \cap \prod_{i=1}^k
  \lbrack z_i-B,z_i+B\lbrack |}{(2B)^k} = \lim_{B \rightarrow \infty}
\frac{ \zeta(k)^{-1}(2B')^k + o((B')^k)}{(2B)^k} = (\zeta(k)\cdot
d^k)^{-1}
\]
\end{proof}

\section{Generalities on unimodular matrices}  \label{unimodular}

Recall that a $n \times n$ matrix $U$ with coefficient in $\Z$ is
unimodular if its determinant is $\pm 1$. Unimodular matrices play a
special role with respect to sets with densities as shown in the next
lemma:

\begin{lem}\label{l2}
Let $S \subset \Z^n$ be a set with density $\delta>0$ and let $V$ be a
$n \times n$ unimodular matrix. Then $V(S) = \left\{ Vx \, : \, x\in
S\right\}$ has a density equal to $\delta$.
\end{lem}
\begin{proof}
Given a cube $\sigma = \prod_{i=1}^n \lbrack z_i-B,z_i+B\lbrack$ in
$\Z^n$, let us count the number of points of the set $V(S)$ that lie
inside $\sigma$. Since $V$ is a bijection, this number is exactly the
number of elements of $S$ inside $V^{-1}(\sigma)$. The map $V^{-1}$ is
linear, and thus $V^{-1}(\sigma)$ is a $n$ dimensional parallelepiped
whose boundary $\partial V^{-1}(\sigma)$ is a union of parallelepipeds
of dimension $n-1$. Let us cover $V^{-1}(\sigma)$ with a disjoint
union of $N$ cubes of side length $B_0$ with $B_0=o(B)$, where $B_0$
is an unbounded function of $B$, e.g. $B_0 = \ln(B)$.
\begin{figure}[htbp]
  \begin{center}
    \includegraphics[width=8cm]{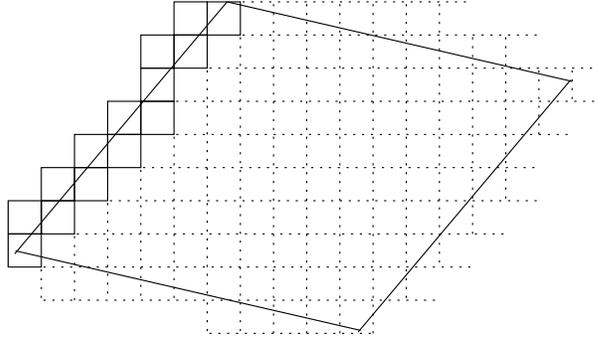}
    \caption{The border effect on  $V^{-1}(\sigma)$}
  \end{center}
\end{figure}
We would like to have an estimation for $N$. Since $V$ is unimodular,
the volume of $V^{-1}(\sigma)$ is $(2B)^n$ (in terms of Lebesgue
measure). The disjoint union of cubes of side length $B_0$ meeting the
border has a total volume which behaves like $O(B_0B^{n-1})$. Taking
into account this border effect, we therefore have (see Figure 1)
\[
N= \frac{(2B)^n + O(B_0B^{n-1})}{B_0^n}.
\]
Each of the cubes of side length $B_0$ contains $\delta B_0^n +
o(B_0^n)$ points of $S$. Therefore,
\begin{eqnarray*}
|S \cap V^{-1}(\sigma)| & = & N \cdot (\delta B_0^n + o(B_0^n)) 
+  O(B_0B^{n-1}) \\
 & = & \delta (2B)^n + O(B_0B^{n-1}) +  \left( (2B)^n +
  O(B_0B^{n-1}) \right)  \cdot  \frac{o(B_0^n)}{B_0^n}.
\end{eqnarray*}
Finally, using the conditions on $B_0$, we see that 
\[
\lim_{B \rightarrow \infty} \frac{|V(S) \cap \sigma|}{(2B)^n} =
\lim_{B \rightarrow \infty} \frac{|S \cap V^{-1}(\sigma)|}{(2B)^n} = \delta
\] 
which finishes the proof of the lemma.
\end{proof}

The previous lemma can be used to prove that unimodular matrices keep
invariant the density of vectors with entries of given greatest common
divisor. More precisely, we have the following proposition.

\begin{pr}\label{pr2}
Let $U \in \Z^{n \times n}$ be a unimodular matrix and $d
\in \N^*$. Then, for any integer $2\leq k \leq n$,  
\[
\D(x \in \Z^n \, : \, \gcd((Ux)_i\, : \, i=n-k+1,\ldots,n))=d) =
(\zeta(k)\cdot d^k)^{-1}.
\] 
\end{pr}

\begin{proof}
Consider $S=\{(y_1,\ldots,y_n) \in \Z^n \, : \, \gcd(y_i \, : \,
i=n-k+1,\ldots n)=d\}$. Because of Lemma \ref{l1} above, one readily
checks that the set $S$ has a density equal to $(\zeta(k)\cdot
d^k)^{-1}$. The result follows by applying Lemma \ref{l2} with $U=V^{-1}$ to
$S$ since 
\[
\{ x \in \Z^n \, : \, \gcd((Ux)_i\, : \,
i=n-k+1,\ldots,n))=d\} = \{  x \in \Z^n \, : \, Ux \in S \} = V(S).
\]
\end{proof}

\section{Distribution of Hermite normal forms}\label{main}

We start this section by noticing that the pivots $h_{ij_i}$ of the HNF
of a matrix $A$ are determined by the greatest common divisor of the
$i \times i$ minors of the matrix that consists in the $j_l^{th}$
columns of $A$, for $l=1,\ldots,i$. This is true because these
$\gcd$'s are left invariant when $A$ is multiplied on the left by any
unimodular matrix, and because the $\gcd$ of the $i \times i$ minors
of the matrix that consists in the $j_l^{th}$ columns of the HNF of
$A$, for $l=1,\ldots,i$, is precisely equal to $\prod_{l=1,\ldots,i}
h_{ij_i}$ (all the other determinants are zero due to the shape of the
HNF of $A$). Note that when $j_i = i$ the above minors are simply the
$i \times i$ minors of the first $i$ columns of $A$.

This property can be used as a basis of a basic algorithm to compute
the HNF of $A$. We start by computing the greatest common divisor
$h_1$ of the entries of the first non-zero column of $A$. Using the
extended Euclidean algorithm we can express $h_1$ as a linear
combination of the entries of the column. In a matrix form, this means
that there exists a sequence of row operations, i.e., there exists a
unimodular matrix $U_1$, such that the first non-zero column of $U_1 A$
is $\lbrack h_1, 0, \ldots, 0\rbrack^t$. This process can be repeated
recursively as follows. There exists a unimodular matrix $U_k$ such
that the first $k$ columns of $U_k A$ form a matrix in HNF, as follows:
\begin{equation}\label{eq1}
U_k A = \left[
\begin{array}{cccccccc}
* &  \cdots & * & * & * & * & * & *\\
0 &  \ddots & * & * & * & * & * & *\\
 &  0    & * & * & * & * & * & *\\
 &      & 0 & x_1 & * & * & * & *\\
 & & \vdots & \vdots &  &  &  & \vdots\\
 &       & 0 & x_s & * & * & * & *\\
\end{array}
\right].
\end{equation}

Let $\lbrack *,\ldots,*, x_1, \ldots, x_s\rbrack^t$ be the $(k+1)$-th
column of $U_k A$.  If $x_i=0$, $i=1,\ldots,s$, then this column is
disregarded. The next column for which one of the $x_i$ is non-zero is
selected. Using the previous remark, the next pivot $h_{ij_i}$ is
given by $h_{ij_i} = \gcd (x_i)$, and using elementary row operations,
there exists a unimodular matrix $U_{k+1}$ such that the corresponding
column of $U_{k+1} A$ is $\lbrack *,\ldots,*, h_{ij_i}, 0, \ldots,
0\rbrack^t$. Appropriate elementary row operations can modify
$U_{k+1}$ and force the $*$ elements of the column to satisfy the
conditions of the HNF, i.e., to belongs to $\lbrack 0, h_{ij_i}
\lbrack$. At the end of the process, the resulting matrix is clearly
in HNF and must therefore be $\HNF(A)$. This algorithmic approach will
be useful in the proof of Proposition \ref{pr1} below. The key point
is that we can construct the HNF of $A$ column after column, from
left to right, via a sequence of left multiplications by unimodular
matrices.

In order to simplify the statement of our results, let us use the
following notation. For any $n \times m$ matrix $A$, the diagonal
$\diag(A)$ of $A$ is the list of elements
$(a_{ii})_{i=1,\ldots,\min(n,m)}$. For given $n$ and $m$, if
$d_1,d_2,...,d_k$ are integer, we write
\[
\Delta_{d_1,d_2,...,d_k} = \left\{A  \in \Z^{n \times m} \, : \,
  \diag(\HNF(A))=(d_1,...,d_k,*,...,*) \right\}
\]
whenever $k\leq \min(n,m)$.

\begin{pr}\label{pr1}
Let $n,m,k$ be positive integers and let $d_1,d_2,...,d_k \in \N$ .
\begin{enumerate}
\item Suppose $n,m,k$ satisfy $k\leq m$ if $m<n$ and $k< n$
otherwise. If $d_k=0$,
\[
\D\left( \Delta_{d_1,d_2,...,d_k} \right)= 0.
\]
If $d_i\neq 0, \forall i=1,...,k$, then
\[
\D\left( \Delta_{d_1,d_2,...,d_k} \right)=
\left( \zeta(n) \cdot \zeta(n-1) ...  \cdot \zeta(n-k+1)
\cdot d_1^n \cdot d_2^{n-1}  \cdot \ldots  \cdot d_k^{n-k+1}  \right)^{-1}.
\]
\item Suppose $n\leq m$ and let $0\leq r<d  \in \N$. Then
\[
\D\left( A \in \Delta_{d_1,d_2,...,d_{n-1},a} \, : \, a \equiv r \mod
d \right)= \frac{1}{d} \cdot \D\left( \Delta_{d_1,d_2,...,d_{n-1}} \right).
\]
\end{enumerate}
\end{pr}

The first point of the previous proposition clearly shows that the
example of HNF given in the introduction (with a 0 in the diagonal)
will only rarely appear. The powers $d_i^{n+1-i}$ appearing in the
expression of the density explain the decreasing expectation to see a
randomly looking $n \times m$ matrix having elements on the top of the
diagonal of its HNF larger than 1. Let us now prove the proposition.

\begin{proof}
We prove by induction on $k$ that the expressions for the density are
valid. For $k=1$, we know that $d_1$ is the greatest common divisor of
the entries of the first column vector. If $d_1=0$, there is only one
possibility for this column, giving a density equal to 0, and when
$d_1\neq 0$, the claim is the result given by Proposition \ref{pr2}
above, when $U$ is the identity and with $k=n$. The induction step is
as follows. For a bound $B$, the number of $n \times (k-1)$ matrices
$A'$ with entries in a cube of side length $2B$ such that
$\diag(\HNF(A'))=(d_1,...,d_{k-1})$ with $d_i\neq 0$, $i=1,\ldots,k-1$
is
\[
(2B)^{n(k-1)} \cdot \left( \zeta(n) \cdot \zeta(n-1) ...  \cdot \zeta(n-k+2)
\cdot d_1^n \cdot d_2^{n-1}  \cdot \ldots  \cdot d_{k-1}^{n-k+2}  \right)^{-1}
+ o(B^{n(k-1)}).
\]
For each of these matrices, there  exists an $n  \times n$ unimodular
matrix  $U$ such  that $UA'$  is  upper triangular  and $\diag(UA')  =
(d_1,...,d_{k-1})$. If $v$ is a vector in $\Z^n$, then 
\[
U \lbrack A'|v \rbrack = \left[
\begin{array}{cccc}
d_1 &  \cdots & * & (Uv)_1\\
0 &  \ddots & * & \vdots\\
 &  0    & d_{k-1} & (Uv)_{k-1}\\
 &      & 0 & (Uv)_k\\
 & & \vdots & \vdots\\
 &       & 0 & (Uv)_n\\
\end{array}
\right].
\]
Based on the algorithmic description of the HNF given earlier, the
diagonal of the HNF of $\lbrack A'|v\rbrack$ is
$(d_1,...,d_{k-1},d_k)$, with $d_k=\gcd((Uv)_i\, : \,
i=k,\ldots,n))$. If $d_k=0$, we have $v = U^{-1}\lbrack *,\ldots,*, 0,
\ldots, 0 \rbrack^t$, i.e., there are only $(2B)^{k-1}$ such $v$ in
any cube $\sigma$ of side length $2B$ and dimension $n$, independently
from $U$.  Since $n>k$, this implies that $\D\left(
\Delta_{d_1,d_2,...,0} \right)= 0$. Suppose now $d_k\neq 0$. Using
Proposition \ref{pr2}, we see that for each $A'$, there are $(2B)^n
\cdot (\zeta(n-k+1)\cdot d_k^{n-k+1})^{-1} + o(B^n)$ such vectors $v$
in $\sigma$. The number of such matrices $\lbrack A'|v\rbrack$ is
then, up to an error of order $o(B^{n(k-1)+n})$,
\[
(2B)^{n(k-1)+n} \cdot \left( \zeta(n) \cdot ...  \cdot \zeta(n-k+2)
\cdot d_1^n \cdot \ldots  \cdot d_{k-1}^{n-k+2}
\right)^{-1} \cdot \left(\zeta(n-k+1)\cdot
d^{n-k+1}\right)^{-1}.
\]
Since $n(k-1)+n=kn$, the claim is correct. This argument can be
continued as long as Proposition \ref{pr2} can be applied, i.e., until
$k\leq m$ if $m<n$ or $k<n$ otherwise. This finishes the proof of the
first statement of the proposition. 

Let us concentrate now on the second statement. There exists a $n
\times n$ unimodular matrix $U$ such that the first $(n-1)$ column of
$UA$ are in HNF. Clearly, if $a=(UA)_{n,n}$ then $a = \HNF(A)_{n,n}$,
i.e., $a=u \cdot \alpha$, where $u$ is the last row of $U$ and
$\alpha$ is the last column of $A$.  For any cube $\sigma$ of side
length $2B$ and dimension $n$, we want to find the number of
$n$-vectors $\alpha$ in $\sigma$ such that $u \cdot \alpha \equiv a
\mod d$. Since the entries of $u$ are coprime, at least one is coprime
to $d$, say $u_i$. For each $(2B)^{n-1}$ choices of $\alpha_j$ in
$\sigma$, $j\neq i$, there are $\frac{2B}{d}+o(B)$ $\alpha_i \in
\sigma$ such that $u_i\alpha_i \equiv a - \sum_{j\neq i} u_j \alpha_j
\mod d$. In other words, the density of the $\alpha$'s is
$\frac{1}{d}$.  The result follows by applying the same counting
argument as before and by using the previous expression of the density
of $\Delta_{d_1,d_2,...,d_{n-1}}$. This finishes the proof of the
proposition.
\end{proof}

\begin{co}\label{co}
Let $d_1,d_2,...,d_{n-1} \in \N^*$. The natural density of $n \times
n$ integer matrices whose HNF has diagonal $\left(d_1, d_2, \ldots,
d_{n-1}, \frac{\det A}{\prod_{i=1..n-1} d_i}\right)$ is
\[
\left( \zeta(n) \cdot \zeta(n-1) ...  \cdot \zeta(2)
\cdot d_1^n \cdot d_2^{n-1}  \cdot \ldots  \cdot d_{n-1}^{2}  \right)^{-1}.
\]
\end{co}

A rectangular $n\times m$ integer matrix (with $n\neq m$) is called
unimodular if the greatest common divisor of its full rank minors is
1. The natural density of unimodular rectangular $n\times m$ integer
matrices, say with $n \neq m$, has been computed in \cite{maze}, with
the weak definition of natural density presented in Section
\ref{natural_density}.  Proposition \ref{pr1} allows to extend the
result to the stronger natural density defined in this article. With
the material in hand, the proof is straightforward, since a $n\times
m$ integer matrix with $n>m$ is unimodular if and only if its HNF has
only 1's in the diagonal.

\begin{co}\label{co2}
The set of $n\times m$ unimodular integer matrices, with $n>m$, has a
natural density equal to $\left( \zeta(n) \cdot \zeta(n-1) ...  \cdot
\zeta(n-m+1) \right)^{-1} $.
\end{co}

\section{Applications}\label{application}

\subsection{Selection of Random Lattices in Cryptology}\label{rl}

In the following, we discuss the consequences of Proposition \ref{pr1}
above to the various shapes of lattice bases that arise in lattice
based cryptology.

An integer lattice $\mathcal{L}$ is a discrete $\Z$-module of
dimension $n$ in $\R^m$ with $\mathcal{L} = \Z b_1 + \ldots + \Z b_n$,
where $b_i\in \Z^m$ and $\vol (\mathcal{L})=\det(\lbrack b_i\cdot
b_j \rbrack_{i,j})^{1/2} \neq 0$. A matrix $B$ whose row vectors $b_i$ are
independent and generate $\mathcal{L}$ is called a basis of the
lattice. Any matrix $B'=UB$ with $U$ unimodular is a basis of
$\mathcal{L}$. We refer the reader to, e.g., \cite[Chapter 3]{nguyen}
and \cite{silverman} for the use of lattices in cryptology. Several
types of lattice bases naturally appear in lattice based
cryptology. Among them, we find the knapsack $n\times (n+1)$ bases
$(a)$, the NTRU $2n \times 2n$ bases $(b)$ and the so-called random
$n\times n$ lattice basis $(c)$.
\[
\begin{array}{ccc}
  \left[
    \begin{array}{cc}
      I_n & x \\
    \end{array}
    \right]
  &
  \left[
    \begin{array}{cc}
      I_n & H_n \\
      0_n & qI_n
    \end{array}
    \right]
    &
  \left[
    \begin{array}{cc}
      I_{n-1} &x \\
      0 &q
    \end{array}
    \right]\\
 (a) & (b) & (c)
\end{array}
\]

A direct consequence of the previous proposition is that the density
of integer matrices $A$ with HNF of the form $(a)$ is 0. The density
of integer matrices $A$ with HNF of the form $(c)$ is given by $\left(
\zeta(n)\cdot \ldots \cdot \zeta(2) \right)^{-1}$. Since $\zeta(n)$
converges rapidly towards 1, the above density converges rather fast
to the limit $\dd$ with $ \dd = \left( \prod_{j=2}^{\infty} \zeta(j)
\right)^{-1} =0.43575707677...$. This translates into the facts that
the random lattice bases of type $(c)$ have a positive density in the
set of lattices with corresponding dimension. The strict positivity of
this density has been know since the work of Goldstein and Mayer
\cite{goldstein} (see also \cite{buchmann} for an elementary
proof). This density being equal to $\dd$, this shows that the process
of selecting random lattice by selecting random row matrices of type
$(c)$ and large determinant $q$ covers almost $44 \%$ of all possible
cases of randomly looking matrices.

In the case of NTRU bases $(b)$, $q=2^s$, where $s$ is a small integer
and Proposition \ref{pr1} suggests that the density of such lattice
bases is roughly equal to $\dd \cdot 2^{-N}$, with
$N=\frac{n^2}{2}s$. Here again, the density is strictly positive, but
much smaller than in the random case $(c)$.

\subsection{Distribution of $\gcd(\det([A|x],\det[A|y]))$}

Using the weak notion of density presented in Section
\ref{natural_density}, Hafner, Sarnak and McCurley have computed the
probability that two randomly looking $n \times n$ matrices are
coprime \cite{hafner}.  The situation where the randomly looking
matrices differs in one column only turns out to be interesting as
well.

Let $A$ be a randomly looking $n\times (n-1)$ integer matrix and $x,y$
be two randomly looking $n$-vectors. The distribution of the greatest
common divisor $g=\gcd(\det([A|x],\det[A|y]))$ has been used in order
to predict the behavior of fast algorithms that compute the HNF of an
integer matrix, see \cite{miccancio,pernet}.  Miccacio and Warinschi
\cite{miccancio} notice that $g$ is ``typically very small for
randomly chosen matrices'', and Pernet and Stein \cite{pernet}, based on
numerical simulation, provide an histogram of the distribution of the
$g$'s. We propose here to exactly compute this distribution based on
the natural density distribution of Proposition \ref{pr1}. Suppose
$\diag(\HNF(A)) = (d_1,\ldots,d_{n-1})$, with $UA = \HNF(A)$, $U$
unimodular. Then
\begin{eqnarray*}
g & = & \gcd(\det([A|x],\det[A|y])) = \gcd(\det([UA|Ux],\det[UA|Uy]))\\
  & = & \prod_{i=1}^{n-1} d_i \cdot \gcd(u \cdot x,u \cdot y)
\end{eqnarray*}
where $u$ is the last row of $U$ and $u \cdot x$ (resp. $u \cdot y$)
is the scalar product of $u$ and $x$ (resp. $y$). Note that since $U$
is unimodular, we have $\gcd(u_i)=1$. The natural distribution of
$\gcd(u \cdot x,u \cdot y)$ in such a case can be computed as
follows. The reader will readily check that for any given modulus $t$,
the distribution of $(u \cdot x \mod t,u \cdot y \mod t)$ is uniform
in $\left(\Z/t\Z\right)^2$. This means that the proportion of pairs
$(u \cdot x,u \cdot y)$ that are divisible by $d$ is $d^{-2}$, and
among them, the proportion of pairs $((u \cdot x)/d,(u \cdot y)/d)$
that are not $(0,0)$ modulo a finite set of prime $P$ is $\prod_{p\in
  P} 1-p^{-2}$.  Since $\gcd(u \cdot x,u \cdot y)=d$ if and only if $d
| u \cdot x$ and $d | u \cdot y$, and for each prime $p$, $p$ cannot
divides $\frac{u \cdot x}{d}$ and $\frac{u \cdot y}{d}$ at the same
time, it appears that the proportion of pairs $(x,y)$ such that this is
true is given by the limit
\[
\frac{1}{d^2} \prod_{p \in \PP} \left(1-p^{-2}\right) = \left(d^2
\zeta(2) \right)^{-1}.
\]
This heuristic approach can be made rigorous by using the methods used in
Section \ref{natural_density} and the localization methods presented
in \cite{maze}. Finally, the natural density $D_n(g)$ of $n\times
(n-1)$ integer matrices $A$ and $n$-vectors $x,y$ such that
$g=\gcd(\det([A|x],\det[A|y]))$ is given by
\begin{equation}\label{dn}
D_n(g)  =  \frac{1}{\zeta(2) \cdot \prod_{k=2}^n \zeta(k)}
\sum_{d_1\cdot \ldots \cdot d_n=g} \frac{1}{d_1^n\cdot
  d_2^{n-1}\cdot\ldots\cdot d_{n-2}^3\cdot d_{n-1}^2 \cdot
  d_{n}^2}.
\end{equation}
If $\sigma_{-k} (g)=\sum_{d|g} d^{-k}$, then using the Dirichlet's
convolution product $*$ of arithmetic functions, we obtain
\[
D_n  =  \frac{1}{\zeta(2) \cdot \prod_{k=2}^n \zeta(k)} \cdot \left(
\sigma_{-n} *\sigma_{-n+1} *  \ldots *\sigma_{-3} *\sigma_{-2} *  \sigma_{-2}
\right) .
\]
Since the Dirichlet series associated to $\sigma_{-k}$ is $\zeta(s+k) \cdot
\zeta(s)$ (see e.g. \cite{tenenbaum}), i.e., $\sum_{g\geq1}
\frac{\sigma_{-k}(g)}{g^s} = \zeta(s) \cdot \zeta(s+k) $, the
Dirichlet series associated to $D_n$ is given by
\[
\sum_{g\geq1} \frac{D_n(g)}{g^s} = \left(\zeta(s)\right)^n \cdot
\frac{\zeta(s+2)}{\zeta(2)} \cdot \prod_{k=2}^n
\frac{\zeta(s+k)}{\zeta(k)} \;, \;\; \Re(s) >1.
\]
If we write $f_n = \zeta(2) \cdot \prod_{k=2}^n \zeta(k) \cdot D_n$,
Equation (\ref{dn}) above shows that the arithmetic function $f_n$ is
multiplicative, i.e., $f_n(gh)=f_n(g)\cdot f_n(h)$ when
$\gcd(g,h)=1$. It is therefore sufficient to compute $f_n(p^\alpha)$
for $p\in \PP$, $\alpha \geq 1$ in order to determine $D_n$
explicitly. Equation (\ref{dn}) with $d_1=p^i$ gives
\[
f_n(p^\alpha) = \sum_{i=0}^{\alpha} \frac{1}{p^{ni}} f_{n-1}(p^{\alpha-i}),
\]
together with $f_1(p^\alpha)=\frac{1}{p^{2\alpha}}$. Based on this
recurrence relation, we can compute $f_n$ for the first value of $n$,
e.g., $f_2(p^{\alpha})=\frac{\alpha+1}{p^{2\alpha}}$ and prove that
$f_n$ converges rapidly to a limit function $f$ that satisfies
\[
f(1)  =  1 \;, \;\;
f(p)  =  \frac{2p-1}{p^2(p-1)} \;, \;\;
f(p^2)  =  \frac{3p^3-p^2-2p+1}{p^4(p-1)^2(p+1)}
\]
and in general for all $\alpha \geq 3$,
\[
f(p^\alpha)  = 
\frac{\alpha+1}{p^{2\alpha}} + \frac{\alpha}{p^{2\alpha+1}} +
\frac{2\alpha-1}{p^{2\alpha+2}} + \frac{3\alpha-3}{p^{2\alpha+3}} +
\frac{5\alpha-7}{p^{2\alpha+4}} +
o\left(\frac{1}{p^{2\alpha+5}}\right).
\]
The first values of 
\[
D(g) = \lim_{n \rightarrow \infty} D_n(g) =
\lim_{n \rightarrow \infty} \left( \zeta(2) \cdot \prod_{k=2}^n
\zeta(k) \right)^{-1} f_n(g) = \frac{\dd}{\zeta(2)} \cdot f(g),
\]
where $\dd$ is the constant defined in Section \ref{rl}, are given
via
\[
f(2) = \frac{3}{4} \, , \,
f(3) = \frac{5}{18} \, , \,
f(4) = \frac{17}{48} \, , \,
f(5) = \frac{9}{100} \, , \,
f(6) = \frac{5}{24} \, , \,
f(7) = \frac{13}{276}.
\]
Numerical simulation showed that already for small dimension $n$, say
$n>5$, the above values of $D$ give very good approximations of the
density $D_n$. We end up this section by noticing that even though the
above remark of Miccancio is true, the expected size of
$\gcd(\det([A|x],\det[A|y]))$ is unbounded. Indeed, the real numbers
$D(g)$, $g\in \N^*$, define a probability distribution on $\N^*$,
i.e., $\sum_{g\in \N^*} D(g) = 1$ and since $D(p) > C/p^2$ for some
$C>0$ and $\sum_{p \in \PP} 1/p = \infty$, the expectation of the
positive integers under this distribution law is $\sum_{g\in \N^*}
gD(g)> \sum_{p \in \PP} C/p = \infty$. Notice that $D(1)=
\frac{\dd}{\zeta(2)} = 0.266014...$ which is not far from $30\%$, as
noted in \cite{miccancio}.

\section{Conclusion}

Numerical experiments indicate that for randomly looking integer
matrices, their Hermite normal forms are not uniformly distributed
among the upper triangular matrices. The frequency of apparition of the
different diagonals is highly structured. In this paper, we explain
this phenomenon, and we exactly compute these frequencies in terms of
natural density. On the way, we define a multidimensional extension
of the usual natural density over $\N$. We use this analysis in order
to shed light on the following two different situations where the
expected form of the HNF of randomly looking matrices play a
role. First, the densities of three types of lattice bases that
naturally appear in lattice based cryptology has been computed.
Second, a probability distribution over the positive integer appearing
in some HNF algorithms has been explicitly evaluated.

\end{document}